\documentclass[11pt]{amsart}  
 \usepackage[dvips]{epsfig}
\usepackage{graphicx}
\usepackage{latexsym}
\usepackage{amsfonts,amsmath,amssymb}
\newtheorem{theorem}{Theorem}[section]

\newtheorem{lemma}[theorem]{Lemma}
\newtheorem{definition}[theorem]{Definition}
\newtheorem{corollary}[theorem]{Corollary}
\newtheorem{example}[theorem]{Example}

\author[M. B\'ona]{Mikl\'os  B\'ona}
\title[On  Conjectures of Joel Lewis]{On a Family of 
  Conjectures of Joel Lewis on Alternating Permutations}
\address{\rm M. B\'ona, Department of Mathematics, 
University of Florida,
358 Little Hall, 
PO Box 118105, 
Gainesville, FL 32611--8105 (USA)
}

\date{}

\begin{document}
\maketitle

\begin{abstract} We prove generalized versions of
 some conjectures of Joel Lewis 
on the number of alternating permutations avoiding certain
patterns. Our main tool is the perhaps surprising observation that
a classic bijection on pattern avoiding permutations often preserves
the alternating property.
\end{abstract}

\section{Introduction} The theory of pattern avoiding permutations has
seen tremendous progress during the last two decades. The key definition is
the following. Let $k\leq n$, let $p=p_1p_2\cdots p_n$ be a permutation of
 length $n$, and
let $q=q_1q_2\cdots q_k$ be a permutation of length $k$. We say that $p$
 avoids $q$ if there 
are no $k$ indices $i_1<i_2<\cdots <i_k$ so that for all $a$ and $b$, 
the inequality $p_{i_a}<p_{i_b}$ holds if and only if the inequality $q_a<q_b$
holds. For instance, $p=2537164$ avoids $q=1234$ because $p$ does not contain
an increasing subsequence of length four. See \cite{combperm} for an overview
of the main results on pattern avoiding permutations. 

Recently, there has been an interest to extend the study of pattern avoiding
permutations to alternating permutations. 
A permutation $p=p_1p_2\cdots p_n$ is called {\em alternating} if
 $p_1<p_2>p_3<p_4>\cdots$, that is, 
if $p_i<p_{i+1}$ if and only if $i$ is odd. In \cite{lewis}, Joel Brewster
 Lewis has made a number of interesting
conjectures on the numbers $A_n(q)$ of alternating permutations of length 
$n$ that avoid a given pattern $q$.
In particular, he conjectured that for all positive integers $n$, the
 equalities 
\begin{equation} A_{2n}(1234) = A_{2n}(1243), \end{equation}
and 
\begin{equation} A_{2n}(12345) = A_{2n}(12354) \end{equation}
hold.  In this paper, we prove a general version of these conjectures, showing
 that for all  $n$ and for all $k$,
the equality
\begin{equation} \label{general}
 A_{2n}(12\cdots k) = A_{2n}(12\cdots k(k-1)) \end{equation}
holds. We also explain why the same equality does not hold if $2n$ is 
replaced by $2n+1$. On the other hand, a slight modification of our method 
will  show that
for all $n$ and for all $k$, the equality 
\begin{equation} \label{modgeneral} A_{n}(12\cdots k)=A_{n}(21\cdots (k-1)k)
\end{equation} holds. The special case of $k=4$, that is, 
the equality $A_{n}(1234)=A_{n}(2134)$ was conjectured by Joel Lewis
in \cite{lewis}.

\section{A classic bijection}

In this section, we review a classic bijection of Julian West that will be
 useful for us.  We point out that in this section, our permutations
do {\em not} have to be alternating. One crucial definition is the following. 

\begin{definition}
The {\em rank} of an entry of a permutation is the length of the longest 
increasing subsequence that ends in that entry. 
\end{definition}

For instance, in $p=3526174$, entries 3, 2, and 1 are of rank one, entries 5 
and 4 are of rank two, entry 6 is of rank
three, and entry 7 is of rank 4. It is straightforward to prove that entries 
of the same rank always form a decreasing
subsequence.
It is also easy to see  that if a permutation $p$ avoids the 
 increasing pattern $12\cdots k$, then 
all entries of $p$ have rank $k-1$ or less. 

For any permutation pattern $q$, let $S_n(q)$ denote the number of
 permutations of length $n$
(or, in what follows, $n$-permutations) that avoid the pattern $q$. 

\begin{lemma} \label{west} \cite{West}
Let $k\geq 3$ be an integer. Then for all positive integers $n$, 
the equality $S_n(12\cdots k)=S_n(12\cdots k(k-1))$ holds. 
\end{lemma}

Note that in the special case of $k=3$, the equality of the lemma reduces
to $S_n(123)=S_n(132)$, and the proof we are going to present below reduces
to the classic Simion-Schmidt bijection \cite{simsch}.  

\begin{proof} We construct a bijection $f$ from the set $X_n$ of all
 $12\cdots k$-avoiding $n$-permutations to the set
$Y_n$ of $12\cdots k(k-1)$-avoiding $n$-permutations. 

Let $p\in X_n$. In order to obtain $f(p)$, leave all entries of $p$ that are
 of rank $k-2$ or less in their place.
Rearrange the entries of rank $k-1$ of $p$ as follows. Let $P$ be the set of
 positions of $p$ in which an entry of 
rank $k-1$ is located, and let $R$ be the set of entries of $p$ that are of
 rank $k-1$. Now fill the positions of $P$ with
the entries in $R$ from left to right, so that each position $i\in P$ is filled with the {\em smallest} entry $r$ of $R$  that has not been placed yet and that is larger than the closest entry of rank $k-2$ on the left of position $i$. Let $f(p)$ be
the obtained permutation. Note that $f(p)$ avoids $12\cdots k(k-1)$ since the
existence of such a pattern in $f(p)$ would mean that the last two entries
of that pattern were not placed according to the rule specified above. 

It is easy to see that this definition always enables us to create $f(p)$.
  Indeed, the very existence of $p$ shows that there is at least one way to assign the
 entries of $R$ to the positions in $P$ so that each of these entries will
 have rank $k-1$ or higher. Putting the smallest eligible entry in the
 leftmost available position can only push other entries of $R$ back, which
 will not 
decrease their rank.

Note that if entry $p_i$ of $p$ was of rank $k-2$ or less, then $p_i$ did not
 move in the above procedure, and the 
rank of $p_i$ did not change. If $p_i$ was of rank $k-1$, then $p_i$ may have
 moved, and the rank of $p_i$ as an entry
of $f(p)$ is $k-1$ or higher.  

In order to see that $f$ is a bijection, we show that it has an inverse. Let $q\in Y_n$.  The unique preimage $f^{-1}(q)$ can then be obtained by keeping all entries
of $q$ that are of rank $k-2$ or less fixed, and placing all the remaining entries (whose set is $R$)  in the remaining slots in decreasing
order. It is easy to see that this can always be done, and that in their new positions, each element of $R$ will have
rank $k-1$. Indeed,  fill the available slots from left to right with available elements of $R$ as follows.
In each step, move the largest element of $R$ that has not been placed yet into the leftmost available slot $j$, and
 move each element of $R$ that has been weakly on the right of $j$ one notch to the right. Then in each step, each
 slot $i$ either contains an entry that is larger than what was in $i$ before, or  an entry that
was on the left of $i$ before. Both of these steps result in the new entry in position $i$ having rank $k-1$.
\end{proof}

\begin{example} Let $k=4$. Then $f(893624751)=893624571$. Indeed, the only entries of rank three in
893624751 are 7 and 5, so $f$ rearranges them so that each spot is filled with the smallest entry larger than 
the closest entry of rank two on the left of that spot (in this case, the entry 4).
\end{example}

\section{Alternating Permutations}

Now we turn our attention to alternating permutations, and prove the results announced in the
introduction.

\begin{theorem} \label{firstmain}
Let $k\geq 3$ be an integer, and let $n$ be an {\em even} positive integer.
Then we have $A_n(12\cdots k)=A_n(12\cdots k(k-1))$. 
\end{theorem}

\begin{proof}
We claim that the fact that $n$ is even implies that the bijection $f$ of Lemma \ref{west} preserves the
alternating property. In other words, if $p$ is an alternating permutation of length $n$ that avoids $12\cdots k$,
then $f(p)$ is an alternating permutation of length $n$ that avoids $12\cdots k(k-1)$, and vice versa, that is,
if $q$ is an alternating permutation of length $n$ that avoids the pattern $12\cdots k(k-1)$, then $f^{-1}(q)$ is an 
alternating $n$-permutation that avoids $12\cdots k$. 

Let $p=p_1p_2\cdots p_n$ be an alternating, $12\cdots k$-avoiding $n$-permutation, where $n$ is an even positive integer.
 Call the entries of $p$ that
are larger than both their neighbors {\em peaks} and call the entries of $p$ that are smaller than both of their
neighbors {\em valleys}. Let us also say that $p_1$ is a valley and $p_n$ is a peak.  It is clear that all entries of $p$
that are of rank $k-1$ must be peaks. Indeed, {\em because $n$ is even}, all valleys in $p$ are followed by a larger
entry, so a valley of rank $k-1$ would have to be followed by a peak of rank $k$ or higher, which is a contradiction. 

Now let us apply the map $f$ of Lemma \ref{west} to our permutation $p$. As we have seen, that map keeps 
entries of rank $k-2$ or less fixed; it only moves entries of rank $k-1$, which are all peaks. Therefore, in order 
 to prove that
$f(p)$ is an alternating permutation, it suffices to show that if $f$ displaces entry $p_i$ of $p$, that entry $p_i$
will be a peak in $f(p)$.

So let us assume that $f$ moves the peak entry $p_i$ of $p$ to a new position. As all valleys are fixed by $f$,
that new position  is necessarily between two valleys, say $a$ and $b$. We need to prove that $p_i>a$ and $p_i>b$.
(If the new position of $p_i$ is at the very end of $f(p)$, then we only need to prove that $p_i>a$.)

By definition, $p_i$ is {\em larger} than the closest entry $y$ of rank $k-2$ on the left of its new position. If $a=y$,
then this means that $p_i>a$. Otherwise, $a$ is of rank $j\leq k-3$. This implies that $a<y$, otherwise the rank of
$a$ would be higher than the rank of $y$. So $a<y<p_i$, and therefore, $p_i>a$ again. 

Similarly, $b<y$, otherwise $b$ would be of rank at least $k-1$ since $y$ is of rank $k-2$. As $p_i>y$, it follows 
that $p_i>b$, proving our claim that $p_i$ is a peak in $f(p)$. This implies that $f(p)$ is an alternating permutation,
since its entries in even positions are all peaks. 

Now let $q$ be an alternating, $12\cdots k(k-1)$-avoiding $n$-permutation. Consider $f^{-1}(q)$, where $f^{-1}$ is the
 inverse of the bijection $f:X_n\rightarrow
Y_n$, as defined in the proof of Lemma \ref{west}. As we saw in the proof of
 that lemma, $f^{-1}(q)$ is obtained from $q$ by rearranging the entries of $q$ that are of rank $k-1$ or
higher in decreasing order. It is easy to see that all these entries are peaks in $q$. Indeed, if $w$ were a valley
of rank $k-1$ or higher in $q$, then there would be a increasing subsequence $w_1w_2\cdots w_{k-2}w$ in $q$.
If the entry immediately preceding $w$ is $v$, then this would mean that $w_1w_2\cdots w_{k-2}vw$ is
a $12\cdots k(k-1)$-pattern in $q$, which is a contradiction. 

So $f^{-1}$ simply permutes some peaks of $q$ among themselves. Therefore,
 in order to prove that
$f^{-1}(q)$ is alternating, it is again sufficient to prove that each peak
 $q_i$ that is displaced by $f^{-1}$ is
a peak in its new position.  Let us say that $f^{-1}$ moves $q_i$ into a new
 position, where its new neighbors
will be $s$ and $t$. If $q_i$ is larger than the old peak entry $Q$ that was
 between $s$ and $t$ before, 
then of course $q_i$ is a peak in $f^{-1}(q)$. If not, that means that $f^{-1}$
 moved $q_i$ to the {\em right}.
However, that means that $q_i$ must be larger than both $s$ and $t$, otherwise
 one of $s$ and $t$ would have
rank $k$ or more in $q$. Indeed, $q_i$, which is an entry of rank $k-1$, would
be on their left in $q$. That would be a contradiction, 
since $s$ and $t$ are not peaks of $q$, so they are of rank less than $k-1$.   
\end{proof}

\begin{example} Let $n=8$, let $k=4$, and let $p=47581623$. Then the only entries of rank three in $p$ 
are 8, 6 and 3. Rearranging them as described above yields the alternating permutation
$f(p)=47561823$.
\end{example}   

A careful look at the above proof reveals which parts of the argument will carry
over to the case of {\em odd} $n$, and which parts will not. The proof of the fact that 
if $q_i$ is a peak in $q$, then $f^{-1}$ moves $q_i$ into a position where $q_i$ will be a 
peak again, did not use the fact that $n$ was even. So $f^{-1}$ maps alternating permutations
to alternating permutations, even if $n$ is odd. However, if $n$ is odd, and $p$ is alternating,
and $12\cdots k$-avoiding,
then $f(p)$ will not be alternating if and only if there is a entry of rank $k-1$ in $p$ that is not
a peak. It is easy to see that that happens precisely when the {\em last} entry of $p$ is of rank $k-1$,
like in $p=23154$. So we have proved the following Corollary.

\begin{corollary} Let $n$ be an odd positive integer. Then the inequality
\[A_n(12\cdots k) \geq A_n(12\cdots k(k-1))\]
holds. Furthermore, $A_n(12\cdots k)-A_n(12\cdots k(k-1))$ is equal to the
 number
of alternating, $12\cdots k$-avoiding $n$-permutations whose last entry is 
of rank $k-1$. 
\end{corollary}

Finally, we use a slightly modified version of our argument to prove the
 following
theorem. 

\begin{theorem} 
Let $n$ be any positive integer. Then for all $k$, we have
\[A_n(12\cdots k)=A_n(21\cdots k).\]
\end{theorem} 

\begin{proof}
The proof is similar to the proof of Theorem \ref{firstmain}. First, we 
construct a bijection $g$ that proves the equality 
\[S_n(12\cdots k) = S_n(213\cdots k)\]
for every $n$. To this end, let us say that an entry of a permutation is of
{\em co-rank} $i$ if the longest increasing subsequence {\em starting} in that
entry has length $i$. If $p\in X_n$, then we define $g(p)$ as follows. 
Let $g$ keep all entries of $p$ that are of co-rank $k-2$ or less fixed. 
Fill the remaining slots with the remaining entries from right to left, so that
in position $j$, we always put the {\em largest} remaining entry that is 
smaller than the closest entry of co-rank $k-2$ on the right of $j$. 
Then $g$ is a bijection from $X_n$ to the set $Z_n$ of $213\cdots k$-avoiding
permutations of length $n$ as can be proved in a way analogous to the proof
of Lemma \ref{west}. 

Next, we claim that $g$ preserves the alternating property.  In order to prove
this, we point out that if $p\in X_n$, then entries of $p$ of co-rank $k-1$
 are necessarily valleys, since
if they were peaks, they would be immediately preceded by a valley, which would
have co-rank at least $k$, a contradiction. {\em Note that unlike in the proof
of Theorem \ref{firstmain}, there is no need for parity restrictions here.} 
The rest of the proof is analogous to that of Theorem \ref{firstmain}.
\end{proof}

\section{Further directions}
As we mentioned, there is a large collection of conjectures in \cite{lewis} claiming that $A_n(q)=A_n(q')$ for
some patterns $q$ and $q'$. Some of these conjectures are for all integers $n$, some others for integers of a given
parity. In many cases, the corresponding equalities $S_n(q)=S_n(q')$ are known to be true. As the results of this
paper show, sometimes the bijection that proves an equality for {\em all} pattern avoiding permutations preserves
the alternating property, and hence can be used to prove the corresponding equality for alternating permutations.
The question is, of course, exactly when  we can do this.

\end{document}